\documentclass[12pt]{article}
\usepackage{amsmath}
\usepackage{amssymb}
\usepackage{amsfonts}
\usepackage{amsthm}
\usepackage{graphicx}
\usepackage{subfigure}

\usepackage{xcolor}

\usepackage{hyperref}

\usepackage{bbm}

\newtheorem{thm}{Theorem}[section]

\newtheorem{prop}[thm]{Proposition}

\theoremstyle{remark}

\def\d{\partial}

\title{{\bf Probabilistic Non-asymptotic Analysis of Distributed Algorithms}}
\author{Nicolas Champagnat\footnote{Inria and Institut Elie Cartan, Universit\'e de Lorraine, Nancy, France} , Ren\'e Schott\footnote{Institut Elie Cartan and LORIA, Universit\'e de Lorraine, Nancy, France} and Denis Villemonais\footnote{Institut Elie Cartan, Universit\'e de Lorraine, Nancy, France}}
\date{}

\begin{document}
\maketitle

\begin{abstract}
We present a new probabilistic analysis of distributed algorithms. Our approach relies on the theory of quasi-stationary
distributions (QSD) recently developped by the first and third authors \cite{CV, CV2, CV3}.
We give properties on the deadlock time and the distribution of the model before deadlock, both for discrete and diffusion models.
Our results are non-asymptotic since they apply to any finite values of the involved parameters (time, numbers of resources, number
of processors, etc.) and reflect the real behavior of these algorithms, with potential applications to deadlock prevention, which are
very important for real world applications in computer science.
\end{abstract}

 \section{Introduction}
Today's distributed systems involve a huge (but finite) number of processors sharing common resources (i.e. are massively parallel).
These systems are inherently fragile. For example, if a processor is running out of memory it can stop the whole system and deadlock
may appear. Usually, analysis of distributed systems leads to asymptotic results (where the parameters of interest tend to infinity)
and characterization of limit laws (of large numbers, central limit theorems, etc.). But infinity does not exist in computer science!
Data structures have large but finite dimension, even the most efficient computer is not able to realize an infinite number of operations. Asymptotic results are, therefore, not very useful. Non-asymptotic results are true for any value of the parameters of interest (time, memory area, etc.) and are, therefore, closer to real world applications in computer science.
The purpose of this paper is to present a non-asymptotic analysis of distributed algorithms which works for any finite number of processors and any finite number of resources (a similar study is possible for other types of algorithms).\\ 
Our approach relies on the theory of quasi-stationary distributions (QSD) recently developed by Champagnat, Villemonais et al. \cite{CCV, CV, CV2, CV3, MV}.

The organization of this paper is as follows: simple distributed algorithms are presented in Section~\ref{sec:ex}.
Section~\ref{sec:general-QSD} contains generalities on quasi-stationary distributions in finite state spaces. Our main results are
stated and proved in Section~\ref{sec:SDE}. Simulations are presented  in Section~\ref{sec:simulations}. Section~\ref{sec:ccl} contains
concluding remarks and some further research aspects.


\section{Examples}
\label{sec:ex}

We describe two models of distributed systems with possible deadlocks, which will be studied numerically in
Section~\ref{sec:simulations}.

\subsection{Colliding stacks}
\label{sec:colliding-stacks}
The presentation of this example is based on \cite{Fla}. For pedagogical reasons, we consider only two stacks but, of course, real storage allocation algorithms involve a huge number of stacks.\\
Assume that two stacks are to be maintained inside a shared (contiguous) memory area of a fixed size $m$. A trivial algorithm will let them grow from both ends of that memory area until their cumulative sizes fill the initially allocated storage ($m$ cells), and the algorithm stops having exhausted the available memory. That shared storage allocation algorithm is to be compared to another option, namely allocating separate zones of size $m/2$ to each of the two stacks. This separate storage allocation method will then halt as soon as any one of the two stacks reaches sizes $m/2$. Several measures may be introduced to compare these two schemes. One of them is the number of operations that can be treated by the algorithms under some appropriate probabilistic model. Another interesting measure of the efficiency of the shared allocation that was proposed by Knuth \cite{Knu}, is the size of the largest stack when both stacks meet and the algorithm runs out of storage. Flajolet \cite{Fla} completely analyzed (combinatorially) this problem and thus solved a question posed by Knuth \cite{Knu} (Vol. 1, Exercice 2.2.2.13). Partial results have been obtained earlier by Yao \cite{Yao}, but it appears that covering all cases of the original problem cannot be achieved by an extension of Yao's methods.
As has been noticed since the problem was initially posed by Knuth \cite{Knu}, the natural formulation is in terms of random walks. Here the random walk takes place in a triangle in a $2$-dimensional lattice space: a state is the couple formed with the size of both stacks. The random walk has two reflecting barriers along the axes (a deletion takes no effect on an empty stack) and one absorbing barrier parallel to the second diagonal (the algorithm stops when the combined sizes of the stacks exhaust the available storage).\\
Probabilistic analyses of the colliding stacks problem have been done (under various assumptions) by Louchard and Schott \cite{LS}, Louchard, Schott, Tolley and Zimmermann \cite{LSTZ}, Maier \cite{Mai}, Guillotin-Plantard and Schott \cite{GS}, Comets, Delarue and Schott \cite{CDS,CDS2}.

\subsection{Banker algorithm}
\label{banker}
For simplicity, we restrict the presentation (as for the colliding stacks) to dimension $2$.\\
Consider two customers $P_1$ and $P_2$ sharing a fixed quantity of a given resource $R$ (money, say). There are fixed upper bounds
$m_1$, $m_2$ on how much of the resource each of the customers will need at any time. The banker decides to affect to the customer
$P_i$ ($i=1,2$) the required units only if the remaining units are sufficient in order to fulfill the requirements of $P_j$
($j=1,2;j\neq i$). The situation is modeled by a random walk in a rectangle with a broken corner (i.e. $(x_1, x_2): 0\leq x_1 \leq
m_1, 0\leq x_2 \leq m_2, x_1 + x_2 \leq m $, where the last constraint generates the broken corner). The random walk is reflected on
the sides parallel to the axes and is absorbed on the sloping side.\\
Probabilistic analyses of this algorithm have been presented in \cite{LS, LSTZ, GS} for two customers and in \cite{CDS, CDS2} for $d\in N$ customers.\\
Maier and Schott \cite{MS} proved partial results for $d$ customers and $r\in N$ resources.\\

\subsection{Description of the Model in higher dimension}
\label{sec:polytope}
We consider the interaction of $q$ processes $P_1$, $P_2$, \ldots, $P_q$, each with its own resource needs. We allow the processes to
access to $r$ different, non-substituable resources (i.e. types of memory) $R_1$, $R_2$, \ldots, $R_r$. We model resource
limitations, and define resource exhaustion as follows. At any time $s$, process $P_i$ is assumed to have allocated some quantity
$y_{i}^{j}(s)$ of resource $R_j$, which may take discrete values (as for the random walks of the previous examples) or continuous
values (as for the diffusion processes considered in~\cite{CDS}). Process $P_i$ is assumed to have some maximum need $m_{ij}$ of
resource $R_j$ so that
\begin{equation}
  \label{eq:1}
  0\leq y_{i}^{j}\leq m_{ij} ,\quad\forall s\geq 0.
\end{equation}
The constant $m_{ij}$ may be infinite; if finite, it is a hard limit which the process $P_i$ never attempts to exceed. The resources
$R_j$ are limited:
\begin{equation}
\label{eq:2}
\sum_{i=1}^{q} y_{i}^{j}(s) < m_j,
\end{equation}
so that $m_j -1$ is the total amount of resource $R_j$ available for allocation. Remember that resource exhaustion occurs when some process $P_i$ issues an unfulfillable request for a quantity of some resource $R_j$. Here "unfulfillable" means that fulfilling the request would violate one of the inequalities~\eqref{eq:2}.\\

The state space of the memory allocation system is a convex polytope with faces defined by hyperplanes $H_1,\ldots,H_k$ which are either reflective or absorbing.
\begin{itemize}  
\item For example, we can consider a model with $d$ processors and a single resource with limit $m$, so that the state space of the
  random walk is
  $$
  E=\left\{ x\in [0,a_1]\times\ldots\times[0,a_d] \text{ s.t.\ }\langle x,v\rangle<m\right\},
  $$
  where $v$ is a vector of $\mathbb{R}_+^d\setminus\{0\}$ with nonnegative coordinates and with absorbing state
  $$
  \partial=\left\{x\in [0,a_1]\times\ldots\times[0,a_d] \text{ s.t.\ }\langle x,v\rangle=m\right\}.
  $$
\item General case. For $d$ processors and $r$ resources, 
  \begin{align}
  \label{eq:domain}
  E=\left\{ x\in [0,a_1]\times\ldots\times[0,a_d] \text{ s.t.\ }\langle x,v_j\rangle<m_j,\ \forall j=1,\ldots,r\right\},
  \end{align}
  where $v_j$ are vectors of $(\mathbb{R}_+)^d\setminus\{0\}$ with nonnegative coordinates and $m_j$ is the maximum amount of
  resource $j$, and with absorbing state
  \begin{equation}
    \label{eq:boundary-domain}
    \partial=\left\{x\in [0,a_1]\times\ldots\times[0,a_d] \text{ s.t.\ }\langle x,v_j\rangle=m_j,\ \forall j=1,\ldots,r \right\}.
  \end{equation}
\end{itemize}

\section{Quasi-stationary distributions}
\label{sec:general-QSD}

The goal of this section is to give a short survey on the main results on quasi-stationary distributions for absorbed Markov
processes and their implications on deadlock prevention and analysis. To keep things simple, we focus here on
the case of continuous-time processes taking values in a finite state space, like in Section~\ref{banker}.

We consider a Markov process $(X_t,t\geq 0)$ taking values in a finite state space $E\cup\partial$, where $\partial$ is absorbing, meaning
that $X_t\in\partial$ a.s.\ for all $t\geq\tau_\partial:=\inf\{t\geq 0: X_t\in\partial\}$. We assume that $\tau_\partial<\infty$
a.s., i.e.\ $\partial$ is accessible from any state in $E$. In the context of the colliding stacks and the banker models, $\tau_\partial$ is the deadlock time.
For all $x\neq y\in E\cup\partial$, we denote by $q_{x,y}\geq 0$ the transition rate from $x$ to $y$ and we set as usual
$$
q_{x,x}:=-q_x:=\sum_{y\neq x} q_{x,y}.
$$
We assume in all this section that the matrix $Q:=(q_{x,y})_{x,y\in E}$ is irreducible, so that Perron-Frobenius theorem applies to the
exponential of the matrix $Q$: when $t\rightarrow+\infty$,
\begin{equation}
  \label{eq:PF}
  (e^{tQ})_{x,y}=e^{-\lambda_0 t}v_xu_y+O(e^{-\lambda_1 t}),  
\end{equation}
where $-\lambda_0$ is the spectral radius of the matrix $Q$, $u$ and $v$ are the normalized, positive left and right eigenvector of
$Q$ for the eigenvalue $-\lambda_0$, i.e.\ $uQ=-\lambda_0 u$, $Qv=-\lambda_0 v$, $\sum_x u_x=1$ and $\sum_x u_xv_x=1$, and
$$
-\lambda_1:=\sup_{\lambda\in\text{Sp}(Q),\ \lambda\neq\lambda_0}\text{Re}(\lambda),
$$
where $\text{Sp}(Q)$ is the spectrum of the matrix $Q$ in $\mathbb{C}$ and $\text{Re}(z)$ is the real part of $z\in\mathbb{C}$. Note
that, since the matrix $Q$ is irreducible and sub-conservative in the sense that $Q\mathbf{1}\leq 0$ with at least one negative
coordinate, where $\mathbf{1}$ is the vector of $\mathbb{R}^E$ with all coordinates equal to $1$, Perron-Frobenius theory entails
that $-\lambda_1<-\lambda_0<0$.

The next result was first proved in~\cite{DS} and follows easily from the formulas
$$
\mathbb{P}_x(t<\tau_\partial)=\mathbb{P}_x(X_t\in E)=\sum_{y\in E}(e^{tQ})_{x,y}
$$ 
and
$$
\mathbb{P}_x(X_t=y\mid t<\tau_\partial)=\frac{\mathbb{P}_x(X_t=y)}{\mathbb{P}(X_t\in E)}=\frac{(e^{tQ})_{x,y}}{\sum_{z\in
    E}(e^{tQ})_{x,z}}.
$$

\begin{prop}
  \label{prop:PF}
  There exists a constant $C$ such that, for all $x,y\in E$ and all $t\geq 0$,
  \begin{equation}
    \label{eq:expo-cvgce}
    |\mathbb{P}_x(X_t=y\mid t<\tau_\partial)-u_y|\leq Ce^{-(\lambda_1-\lambda_0)t}    
  \end{equation}
  and the probability measure $u=(u_x,x\in E)$ on $E$ is a quasi-stationary distribution, in the sense that
  \begin{equation}
    \label{eq:def-QSD}
    \mathbb{P}_u(X_t=y\mid t<\tau_\partial)=u_y,\quad\forall y\in E,\ t\geq 0,    
  \end{equation}
  where $\mathbb{P}_u=\sum_{x\in E} u_x\mathbb{P}_x$. In addition, for all $x\in E$ and $t\geq 0$,
  \begin{equation}
    \label{eq:eta}
    |e^{\lambda_0 t}\mathbb{P}_x(t<\tau_\partial)-v_x|\leq Ce^{-(\lambda_1-\lambda_0)t}.
  \end{equation}
\end{prop}

Quasi-stationary distributions like $u$ satisfy general properties as explained for example in~\cite{MV}, summarized in the next
proposition.

\begin{prop}
  \label{prop:pty-QSD}
  When the initial population is distributed according to the quasi-stationary distribution $u$, the absorption time $\tau_\partial$
  is exponentially distributed with parameter $\lambda_0$, i.e.
  $$
  \mathbb{P}_u(t<\tau_\partial)=e^{-\lambda_0 t},\quad\forall t\geq 0,
  $$
  $\tau_\partial$ is independent of $(X_{\tau_\partial-},X_{\tau_\partial})$, where $X_{\tau_\partial-}$ is the position just before
  exit and $X_{\tau_\partial}$ is the exit position. In addition, the joint law of $(X_{\tau_\partial-},X_{\tau_\partial})$ under
  $\mathbb{P}_u$ is given by
  $$
  \mathbb{P}_u(X_{\tau_\partial-}=x,\,X_{\tau_\partial}=y)=\frac{u_xq_{x,y}}{\lambda_0},\quad\forall x\in E, y\in\partial.
  $$
\end{prop}

Note that, as a consequence, the exit position is distributed under $\mathbb{P}_u$ as
$$
\mathbb{P}_u(X_{\tau_\partial}=y)=\frac{\sum_{x\in E}u_x q_{x,y}}{\lambda_0},\quad\forall y\in\partial.
$$
This indeed defines a probability distribution on $\partial$ since
\begin{align*}
  \sum_{y\in\partial}\sum_{x\in E}u_x q_{x,y}=\sum_{x\in E}u_x\left(q_{x,x}-\sum_{y\in E,\ y\neq x}q_{x,y}\right)=\sum_{x\in
    E}u_xq_{x,x}+\sum_{y\in E}u_y(\lambda_0-q_{y,y})=\lambda_0.
\end{align*}
In addition, the position just before exit is distributed under $\mathbb{P}_u$ as
$$
\mathbb{P}_u(X_{\tau_\partial-}=x)=\frac{\sum_{y\in \partial}u_x q_{x,y}}{\lambda_0},\quad\forall x\in E.
$$
This is the quasi-stationary distribution biaised by the exit rate $\sum_{y\in\partial}q_{x,y}$ of the process.

Although the method of proof is quite standard, since the independence between $\tau_\d$ and $X_{\tau_\d -}$ and the last display of
Proposition~\ref{prop:pty-QSD} are not stated in classical references like~\cite{MV}, we give the proof for sake of completeness.

\begin{proof}
  The first property follows from Markov's property and the definition of a quasi-stationary distribution~\eqref{eq:def-QSD}
  \begin{align*}
    \mathbb{P}_u(t+s<\tau_\partial) & =\mathbb{E}_u[\mathbbm{1}_{t<\tau_\partial}\mathbb{P}_{X_t}(s<\tau_\partial)] \\
    & =\mathbb{E}_u[\mathbb{P}_{X_t}(s<\tau_\partial)\mid t<\tau_\partial]\mathbb{P}_u(t<\tau_\partial) \\
    & =\mathbb{P}_u(s<\tau_\partial) \mathbb{P}_u(t<\tau_\partial).
  \end{align*}
  This is the ``loss of memory'' property characterizing exponential random variables, hence $\tau_\partial$ is exponentially
  distributed.\\
Since $\mathbb{P}_u(t<\tau_\partial)=\sum_{x\in E} u_x\sum_{y\in E}(e^{tQ})_{x,y}$, it follows from~\eqref{eq:PF} that
  the parameter of the exponential distribution is $\lambda_0$.

  The independence between $\tau_\partial$ and $X_{\tau_\partial}$ follows from a similar computation: for all bounded measurable
  function $f$ on $E$,
  \begin{align*}
    \mathbb{E}_u(f(X_{\tau_\partial})\mathbbm{1}_{t<\tau_\partial})  & =\mathbb{E}_u[\mathbb{E}_{X_t}(f(X_{\tau_\partial}))\mid
    t<\tau_\partial]\mathbb{P}_u(t<\tau_\partial) \\
    & =\mathbb{E}_u(f(X_{\tau_\partial})) \mathbb{P}_u(t<\tau_\partial).
  \end{align*}
  The independence between $\tau_\partial$ and $X_{\tau_\partial-}$ can be proved exactly the same way.

  Finally, due to the above independence, we have for all $t\geq 0$ and all $x\in E$, $y\in\partial$,
  \begin{align*}
    \mathbb{P}_u(X_{\tau_\partial-}=x,\,X_{\tau_\partial}=y)  & =\mathbb{P}_u(X_{\tau_\partial-}=x,\,X_{\tau_\partial}=y \mid \tau_\partial\leq t) \\
    & =\frac{\mathbb{P}_u(X_{\tau_\partial-}=x,\,X_{\tau_\partial}=y,\,\tau_\partial\leq t)}{\mathbb{P}_u(\tau_\partial\leq t)} \\
    & =\frac{\mathbb{P}_u(X_0=x,\, X_t=y,\, J_1\leq t,\,J_2>t)+O(\mathbb{P}_u(J_2\leq t))}{1-e^{-\lambda_0 t}},
  \end{align*}
  where $(J_i)_{i\geq 1}$ is the sequence of jump times of the process $(X_t,t\geq 0)$. Standard computations for discrete Markov
  processes entail
  \begin{align*}
    \mathbb{P}_u(X_{\tau_\partial-}=x,\,X_{\tau_\partial}=y) & =\frac{tu_xq_{x,y}+O(t^2)}{\lambda_0
      t+O(t^2)}\xrightarrow[t\rightarrow 0]{}\frac{u_x q_{x,y}}{\lambda_0}.
  \end{align*}
  This concludes the proof of the Proposition.
\end{proof}


\section{Main results}
\label{sec:SDE}

We study models with discrete state space as those presented above and diffusion models with boundary conditions as in
Section~\ref{sec:polytope}. We obtain estimates on the QSD in both cases and on the Perron-Frobenius eigenvector (resp.\ Dirichlet
eigenfunction) in the discrete case (resp.\ continuous case).

\subsection{Distributions of exit time and exit position in the finite case}
\label{sec:non-asy}
Propositions 3.1. and 3.2. allow us to give estimates on the distributions of the exit time $\tau_\partial$ and the exit position
$f(X_{\tau_\partial})$ of the process depending on $\lambda_0$ and $\lambda_1$. First,~\eqref{eq:eta} entails that
$$
\mathbb{P}_x(t<\tau_\partial)\sim v_xe^{-\lambda_0 t}\quad\text{when }t\rightarrow+\infty,
$$
hence the distribution of the exit time $\tau_\partial$ has an exponential tail. We can also deduce from~\eqref{eq:eta} estimates on
the expectation of functions of $\tau_\partial$: for example,
$$
\left|\mathbb{E}_x(\tau_\partial)-\frac{v_x}{\lambda_0}\right|\leq\frac{C}{\lambda_1}.
$$
Note that this estimate is accurate provided that $\lambda_1\ll\lambda_0$. This is the regime where the state of the process can be
approximated by the quasi-stationary distribution for intermediate times, as explained in the next proposition.

\begin{prop}
  \label{prop:2}
  For all $x\in E$ and $t\geq 0$,
  \begin{equation}
    \label{eq:estimate-eta}
    e^{\lambda_0 t}+Ce^{-(\lambda_1-\lambda_0)t}-\frac{e^{\lambda_0 t}-1+Ce^{-(\lambda_1-\lambda_0)t}}{u_x}\leq v_x\leq e^{\lambda_0 t}+Ce^{-(\lambda_1-\lambda_0)t}.
  \end{equation}
  In addition, for all $x,y\in E$ and $t\geq 0$,
  \begin{equation}
    \label{eq:estimate-QSD}
    \left|\mathbb{P}_x(X_t=y)-u_y\right|\leq \frac{u_y}{u_x}(1-e^{-\lambda_0 t}+Ce^{-\lambda_1 t})+Ce^{-\lambda_1 t}\left(v_x+2
      u_y+Ce^{-(\lambda_1-\lambda_0)t}\right).
  \end{equation}
\end{prop}

In the case where $\lambda_0\ll\lambda_1$, for $1/\lambda_1\ll t\ll 1/\lambda_0$, recalling that $v_x$ is bounded, we deduce that
$\mathbb{P}_x(X_t=y)\approx u_y$ for all $x,y\in E$.

\begin{proof}
  The upper bound in~\eqref{eq:estimate-eta} follows directly from~\eqref{eq:eta}. For the lower bound, we use the upper bound and
  the fact that $\sum_{x\in E}u_x=\sum_{x\in E}u_xv_x=1$ to obtain
  \begin{align*}
    u_xv_x & =1-\sum_{y\in E,\ y\neq x}u_yv_y \\ & \geq 1-e^{\lambda_0 t}\sum_{y\in E,\ y\neq x}\left(u_y-Ce^{-(\lambda_1-\lambda_0)t}u_y\right)
    \\ & =1-e^{\lambda_0 t}+u_x e^{\lambda_0 t}-Ce^{-(\lambda_1-\lambda_0)t}+Cu_xe^{-(\lambda_1-\lambda_0)t}.
  \end{align*}
  The lower bound in~\eqref{eq:estimate-eta} follows.

  To obtain~\eqref{eq:estimate-QSD}, we combine~\eqref{eq:expo-cvgce},~\eqref{eq:eta} and~\eqref{eq:estimate-eta} in
  \begin{align*}
    & \left|\mathbb{P}_x(X_t=y)-u_y\right| \\
    & \leq \left|\mathbb{P}_x(X_t=y\mid
      t<\tau_\partial)-u_y\right|\mathbb{P}_x(t<\tau_\partial)+u_y\left|\mathbb{P}_x(t<\tau_\partial)-v_x e^{-\lambda_0
        t}\right|+u_y\left|v_x e^{-\lambda_0 t}-1\right| \\
    & \leq Ce^{-(\lambda_1-\lambda_0)t}(v_x e^{-\lambda_0 t}+C e^{-\lambda_1 t})+Cu_y e^{-\lambda_1 t}+Cu_y e^{-\lambda_1
      t}+\frac{u_y}{u_x}\left(1-e^{-\lambda_0 t}+C e^{-\lambda_1 t}\right). \qedhere
  \end{align*}
\end{proof}

We also obtain estimates on the distribution of $(X_{\tau_\partial-},X_{\tau_\partial})$.

\begin{prop}
  \label{prop:unconditioned}
  For all $x,y\in E$, all $z\in\partial$ and all $t\geq 0$,
  \begin{equation}
    \label{eq:unconditioned}
    \left|\mathbb{P}_x (X_{\tau_\partial-}=y,\,X_{\tau_\partial}=z)-\frac{u_yq_{y,z}}{\lambda_0}\right| \leq
    2C\frac{1-e^{-\lambda_0 t}}{u_x}+Ce^{-\lambda_1 t}\left(4+v_x+Ce^{-(\lambda_1-\lambda_0)t}\right).
  \end{equation}
\end{prop}

Again, the inequality~\eqref{eq:unconditioned} gives an accurate estimate on the distribution of
$(X_{\tau_\partial-},X_{\tau_\partial})$ in the case where $\lambda_0\ll\lambda_1$: taking $1/\lambda_1\ll t\ll 1/\lambda_0$, it
follows that $\mathbb{P}_x(X_{\tau_\partial-}=y,X_{\tau_\partial}=z)\approx \lambda_0^{-1} u_yq_{y,z}$ for all $x,y\in E$ and
$z\in\partial$.

\begin{proof}
  The proof makes use of the estimates of Proposition~\ref{prop:PF} and the properties of Proposition~\ref{prop:pty-QSD}: for all
  bounded measurable function $f$ on $E\times\partial$,
  \begin{multline*}
    \left|\mathbb{E}_x \left[f(X_{\tau_\partial-},X_{\tau_\partial})\mathbbm{1}_{t<\tau_\partial}\right]- v_x e^{-\lambda_0
        t}\mathbb{E}_u[f(X_{\tau_\partial-},X_{\tau_\partial})]\right| \\
    \begin{aligned}
      & \leq\mathbb{P}_x(t<\tau_\partial)\left|\mathbb{E}_x \left[f(X_{\tau_\partial-},X_{\tau_\partial})\mid t<\tau_\partial\right]-\mathbb{E}_u
        \left[f(X_{\tau_\partial-},X_{\tau_\partial})\mid t<\tau_\partial\right]\right| \\ & +\mathbb{E}_u
      \left[f(X_{\tau_\partial-},X_{\tau_\partial})\right]\left|\mathbb{P}_x(t<\tau_\partial)-v_x e^{-\lambda_0 t}\right| \\
      & \leq \left(e^{-\lambda_0 t}v_x+Ce^{-\lambda_1 t}\right)C\|f\|_\infty e^{-(\lambda_1-\lambda_0)t}+C\|f\|_\infty e^{-\lambda_1
        t}.
    \end{aligned}
  \end{multline*}
  This entails
  \begin{multline*}
    \left|\mathbb{E}_x \left[f(X_{\tau_\partial-},X_{\tau_\partial})\mathbbm{1}_{t<\tau_\partial}\right]- v_x e^{-\lambda_0
        t}\mathbb{E}_u[f(X_{\tau_\partial-},X_{\tau_\partial})]\right| \\ \leq Ce^{-\lambda_1
      t}\|f\|_\infty\left(1+v_x+Ce^{-(\lambda_1-\lambda_0)t}\right).
  \end{multline*}
  Therefore, it follows from the inequality
  \begin{align*}
    \mathbb{P}_x(\tau_\partial\leq t) & \leq 1-v_xe^{-\lambda_0 t}+Ce^{-\lambda_1 t}
  \end{align*}
  that
  \begin{multline*}
    \left|\mathbb{E}_x
      \left[f(X_{\tau_\partial-},X_{\tau_\partial})\right]-\mathbb{E}_u[f(X_{\tau_\partial-},X_{\tau_\partial})]\right| \\ \leq
    2\|f\|_\infty\left(1-v_xe^{-\lambda_0 t}\right)+Ce^{-\lambda_1 t}\|f\|_\infty\left(2+v_x+Ce^{-(\lambda_1-\lambda_0)t}\right).
  \end{multline*}
  We conclude from~\eqref{eq:estimate-eta} and Proposition~\ref{prop:pty-QSD}.
\end{proof}

The previous results are sharp under the condition $\lambda_0\ll\lambda_1$, which means that absorption takes a long time (the typical
absorption time is $1/\lambda_0$) and the process has a tendency to stay away from the absorbing boundaries. This is for example the
case when the random walk converges to a deterministic process for which the interior of the domain is stable
(see~\cite{FS,CCM16,CCM17}). Our estimates can be applied to the examples of Section~\ref{sec:ex} by numerically computing the
eigenvalues $\lambda_0$ and $\lambda_1$.


\subsection{The multi-dimensional diffusion model}
\label{sec:QSD-diff}

The previous section gave results on absorbed Markov processes in finite state space, like those of the examples of
Section~\ref{sec:ex}. It is also common to model deadlocks replacing discrete random walks with diffusion processes in subsets of
$\mathbb{R}^d$ with partly absorbing and reflecting boundaries, like in the example presented in Section~\ref{sec:polytope}. In this
case, quasi-stationary distributions may still be defined, although asymptotic properties as those of Propositions~\ref{prop:PF}, and
hence estimates as in Proposition~\ref{prop:2}, are harder to obtain. A general criterion for such results was recently obtained
in~\cite{CV}, which has been applied to various classes of stochastic processes in~\cite{CCV, CV, CV2, CV3}. A particular case of
diffusion in a domain delimited by hyperplanes was studied in~\cite{CV3} using non-linear Lyapunov criteria. Howerer, situations with
parts of the boundary being absorbing and other parts reflecting were never studied. We are able to obtain the next result for
general colliding stacks models.

\begin{thm}
  \label{thm:QSD-diff}
  Consider the diffusion process $X$ evolving in $E\cup\d$ as defined in~\eqref{eq:domain} and~\eqref{eq:boundary-domain} with
  $m_j=1$ for $j\in\{1,\ldots,r\}$ arbitrary nonnegative vectors $v_j$, $a_i=+\infty$ for $i\in\{1,\ldots,d\}$, with the hyperplanes
  $x_i=0$, $i\in\{1,\ldots,d\}$ as normal reflecting boundaries and the hyperplanes $\langle x,v_j\rangle =1$ as absorbing. Assume
  that the infinitesimal generator of $X$ is given for all smooth function $f$ vanishing on the hyperplanes $\langle x,v_j\rangle =1$
  and with zero normal gradient on the hyperplanes $x_i=0$, by
  \begin{align}
    \label{eq:inf-gen}
    L f(x)=
    \sum_{i,k=1}^d a_{ik}(x)\frac{\d^2 f}{\d x_i \d x_k}(x)+\sum_{i=1}^d b_i(x)\frac{\d f}{\d x_i}(x),
  \end{align}
  where the matrix $a=(a_{ij})_{1\leq i,j\leq d}$ is symmetric and uniformly elliptic and $b$ is uniformly bounded,
  both H\"older continuous on $E\cup\d$. Then $X$ admits a unique quasi-stationary distribution $\alpha$ and there exist positive
  constants $C,\gamma>0$ such that, for all $t\geq 0$ and all probability measure $\mu$ on $E$,
  \begin{align}
    \label{eq:expo-conv-diff}
    \left\|\mathbb{P}_\mu(X_t\in\cdot\mid t<\tau_\d)-\alpha\right\|_{TV}\leq C e^{-\gamma t}.
  \end{align}
\end{thm}

The conclusion of Theorem~\ref{thm:QSD-diff} is equivalent to Condition~(A) in~\cite{CV}. This has several implications. For
instance, $e^{\lambda_0 t}\mathbb{P}_x(t<\tau_\d)$ converges when $t\rightarrow+\infty$, uniformly in $x$, to a positive, bounded
eigenfunction $\eta$ of $L$ with Dirichlet boundary condition at the absorbing boundary, for the eigenvalue $-\lambda_0$,
characterized by the relation $\mathbb{P}_\alpha(t<\tau_\d)=e^{-\lambda_0 t}$, $\forall t\geq 0$~\cite[Proposition~2.3]{CV}.
Moreover, it implies a spectral gap property~\cite[Corollary~2.4]{CV}, the existence and exponential ergodicity of the so-called
$Q$-process, defined as the process $X$ conditioned to never hit the absorbing part of the boundary~\cite[Theorem~3.1]{CV} and a
conditional ergodic property~\cite{CV2}.

Note that the assumption that the constants $m_j$ are all equal to 1 may be relaxed by applying a linear scaling of coordinates,
possibly different for each coordinate. Such a scaling has no impact on the required assumptions on the coefficients of the
diffusion.


\begin{proof}[Proof of Theorem~\ref{thm:QSD-diff}] 
We start by considering the case where $b_i(x)=0$ whenever $x_i=0$. We shall extend, in a second step, our result to the general case.

We extend the definition of $\sigma$ and $b$ on
\begin{align*}
  {E'=\{x\in\mathbb{R}^d,\ \langle |x|,v_j\rangle<1\ \forall j\in\{1,\ldots,d\}\},}
\end{align*}
where $|x|$ is defined as the vector $(|x_1|,\ldots,|x_d|)$, by symmetry over the hyperplanes $x_i=0$, $i\in\{1,\ldots,d\}$: more
precisely, for all $x\in E$, we set $a(y)=a(x)$ if $|y|=x$ and $b_i(y)=\text{sign}(y_i)b_i(x)$ for all $i=1,\ldots,d$. Since
$b_i(x)=0$ whenever $x_i=0$, the extended coefficients are also H\"older continuous on $E'$. Hence, we can define a diffusion process
$Y$ evolving in $E'$ with infinitesimal generator~\eqref{eq:inf-gen}, absorbed at the boundary of $E'$. It follows from standard
properties of diffusions with normal reflexion on hyperplanes that the process $(|Y^1_t|,\ldots,|Y^d_t|)_{t\geq 0}$ has the same law
as the process $X$ defined in the statement of Theorem~\ref{thm:QSD-diff}. In the proof, we show that~\eqref{eq:expo-conv-diff} holds
true for $Y$, so that it holds true for $X$.

Since $L$ is assumed to be elliptic, there exist two constants $\overline{\lambda}\geq \underline{\lambda}>0$ such that 
$$\overline{\lambda}\|x\|^2\geq \langle x,a(x) x\rangle\geq \underline{\lambda} \|x\|^2,\ \text{ for all }x\in E'.$$ 
In order to prove that Condition~(A) of~\cite{CV} holds true, we use the Lyapunov type criterion proved in~\cite[Proposition
2.7]{CV3}, with the functions
\begin{align*}
\varphi(x)&:=\prod_{\varepsilon\in\{-1,1\}^d}\prod_{j=1}^r (1-\langle \varepsilon v_j,x\rangle)^{\alpha},\\
V(x)&:=\prod_{\varepsilon\in\{-1,1\}^d}\prod_{j=1}^r (1-\langle \varepsilon v_j,x\rangle)^{\beta},
\end{align*}
where $\varepsilon v_j$ is intended as a componentwise product $\varepsilon v_j=(\varepsilon_i v_{j,i})_{1\leq i\leq d}$,
$$
\alpha=\frac{1+\|a\|_\infty\frac{4d^3 2^dr \bar{v}}{\underline{v}^2}}{\underline{\lambda}}\quad\text{and}\quad
\beta=\frac{\underline{\lambda}}{2r 2^d\overline{\lambda}}\wedge \frac{1}{2^d r},
$$
where $\bar{v}=\sup_{1\leq j \leq r,\ 1\leq i\leq d}v_{j,i}$ and $\underline{v}=\inf_{i,j\ s.t.\ v_{j,i}>0} v_{j,i}$. Note that both $V$
and $\varphi$ are smooth functions on $E'$.

\smallskip
Assumptions~2 and~4 in~\cite{CV3} are satisfied because $L$ is, in particular, locally elliptic with H\"older coefficients (the
arguments are detailed in~\cite[Section~4]{CV3}). Assumption~3 in~\cite{CV3} is an easy consequence of the boundedness of $E'$ and of
the uniform ellipticity of~$L$. One easily checks by standard arguments (see for instance~\cite[Section~3]{CCV}) that, for all
$t_0>0$, there exist a constant $A>0$ such that
\begin{align*}
\mathbb{P}_x(t_0 < \tau_\d)\leq A d(x,\d E'). 
\end{align*}
Since $d(x,\d E')=\min_{\varepsilon\in\{-1,1\}^d} (1-\langle \varepsilon v_j,x\rangle)$, we deduce that
\begin{align*}
  \mathbb{P}_x(t_0 < \tau_\d)\leq A\prod_{\varepsilon\in\{-1,1\}^d}\prod_{j=1}^r (1-\langle \varepsilon v_j,x\rangle)^{1/2^d r}\leq A V(x),
\end{align*}
which is Condition~(2.9) of Assumption~1 in~\cite{CV3}.

Hence, it remains to prove that there exist a compact set $K\subset E'$ and positive constants $\delta,c,c',c''>0$ such that,
for all $x\in E'$,
\begin{align}
\label{eq:proof-conclusion}
L\varphi(x)\geq -c\mathbf{1}_{x\in K}\text{ and }
LV+c'\frac{V^{1+\delta}}{\varphi^\delta}\leq c''\varphi.
\end{align}
Indeed, by \cite[Proposition~2.7]{CV3}, this implies that Assumption~1 in~\cite{CV3} is satisfied, so that~\eqref{eq:expo-conv-diff} holds true (by~\cite[Corollary~2.8]{CV3}).

Setting 
$$
f_i(x):= \sum_{\varepsilon\in\{-1,1\}^d}\sum_{j=1}^r \frac{\varepsilon_i v_{ji}}{1-\langle \varepsilon v_j,x\rangle},\quad\forall
i\in\{1,\ldots,d\},
$$
we have, for all $x\in E'$, 
\begin{align*}
L\varphi(x)&=\langle a(x)\nabla,\nabla\rangle \varphi(x)+\langle b(x),\nabla\rangle\varphi(x)\\
&=-\alpha  \langle a(x)\nabla,\varphi(x)f(x)\rangle-\alpha\langle b(x),\varphi(x)f(x)\rangle\\
&=\alpha\varphi(x)\langle a(x)(\alpha f(x)-\nabla),f(x)\rangle-\alpha\varphi(x)\langle b(x),f(x)\rangle\\
&\geq \alpha\varphi(x)\left[\alpha\underline{\lambda}\|f(x)\|^2-\|a\|_\infty\langle\nabla,f\rangle-\sqrt{d}\|b\|_\infty
  \|f(x)\|\right].
\end{align*}
Now, 
\begin{align*}
  \langle\nabla,f\rangle(x) & = \sum_{\varepsilon\in\{-1,1\}^d}\sum_{i=1}^d\frac{\|v_{j}\|^2}{(1-\langle \varepsilon v_j,x\rangle)^2} \\
  & \leq 2^ddr\bar{v}\max_{1\leq j\leq r}\frac{1}{|1-\langle v_j,|x|\rangle|^2}.
\end{align*}
Let $K_0=\{x\in E':\forall j\in\{1,\ldots,r\},\ \langle v_j,x\rangle\leq 1/2\}$. Our goal is to prove that $\langle\nabla,f\rangle(x) \leq
C_0\|f(x)\|^2$ for all $x\not\in K_0$ for an appropriate constant $C_0$. We can assume without
loss of generality that $x\geq 0$, meaning that all its coordinates are nonnegative. We have for all $1\leq i \leq d$
\begin{align*}
  f_i(x) & =\sum_{\varepsilon^{(i)}\in\{-1,1\}^{d-1}}\sum_{j=1}^rv_{ji}\frac{2v_{ji}x_i}{(1-v_{ji}x_i-\langle\varepsilon^{(i)}
    v^{(i)}_j,x^{(i)}\rangle)(1+v_{ji}x_i-\langle\varepsilon^{(i)} v^{(i)}_j,x^{(i)}\rangle)} \\
  & \geq \sum_{\varepsilon^{(i)}\in\{-1,1\}^{d-1}}\sum_{j=1}^r\frac{v^2_{ji}x_i}{1-v_{ji}x_i-\langle\varepsilon^{(i)}
    v^{(i)}_j,x^{(i)}\rangle} \\
  & \geq \sum_{j=1}^r\frac{v^2_{ji}x_i}{1-\langle v_j,x\rangle},
\end{align*}
where $x^{(i)}$ denotes the vector of $\mathbb{R}^{d-1}$ obtained from $x$ by suppressing its $i$-th coordinate. Hence
\begin{align*}
  \|f(x)\|^2 & \geq \sum_{i=1}^d\left(\sum_{j=1}^r\frac{v^2_{ji}x_i}{1-\langle v_j,x\rangle}\right)^2 \\
  & \geq \sum_{i=1}^d\sum_{j=1}^r\frac{v^4_{ji}x^2_i}{|1-\langle v_j,x\rangle|^2}.
\end{align*}
Let $j_0$ be such that $\langle v_{j},x\rangle$ is maximal. In particular, it is larger than $1/2$, and there exists $i_0$ such that
$v_{j_0 i_0}x_{i_0}\geq 1/2d$. Then,
\begin{align*}
  \|f(x)\|^2 & \geq \frac{v^4_{j_0 i_0}x^2_{i_0}}{|1-\langle v_{j_0},x\rangle|^2}\geq \frac{\underline{v}^2}{4d^2}\max_{1\leq
    j\leq r }\frac{1}{|1-\langle v_j,|x|\rangle|^2}\geq \frac{\underline{v}^2}{4d^3 2^d r\bar{v}}\,\langle\nabla,f\rangle(x).
\end{align*}
So by definition of $\alpha$
\begin{align*}
  L\varphi(x)&\geq \alpha\varphi(x)(\|f(x)\|^2-\sqrt{d}\|b\|_\infty \|f(x)\|).
\end{align*}
But $\|f\|$ converges to $+\infty$ when $x\rightarrow \d E'$, so that the first part of~\eqref{eq:proof-conclusion} holds true. 

Let us now prove that second part also holds true. We have
\begin{align*}
LV(x)&=\beta V(x)\langle a(x)(\beta f(x)-\nabla),f(x)\rangle-\beta V(x)\langle b(x),f(x)\rangle\\
&\leq \beta V(x)\left[\beta\overline{\lambda}\|f(x)\|^2-\langle a(x)\nabla,f(x)\rangle +\langle b(x),f(x)\rangle\right].
\end{align*}
Since
\begin{align*}
-\langle a(x)\nabla,f(x)\rangle&= -\sum_{\varepsilon\in\{-1,1\}^d}\sum_{j=1}^r\sum_{i=1}^d \sum_{k=1}^d a_{ik} \frac{\varepsilon_i v_{ji}
  \varepsilon_k v_{jk}}{(1-\langle \varepsilon v_j,x\rangle)^2}\\
&\leq -\underline{\lambda}\sum_{\varepsilon\in\{-1,1\}^d}\sum_{j=1}^r \frac{\|v_j\|^2}{(1-\langle \varepsilon v_j,x\rangle)^2} \\ & =-\underline{\lambda}\langle\nabla,f\rangle(x),
\end{align*}
and, by Cauchy-Schwarz inequality,
\begin{align}
\label{eq:bound-on-f2}
\|f(x)\|^2\leq r2^d\langle\nabla,f\rangle(x),
\end{align}
we deduce from the definition of $\beta$ that
\begin{align*}
LV(x)&\leq \beta V(x)\left[-\frac{\underline{\lambda}}{2}\langle\nabla,f\rangle(x)+\langle b(x),f(x)\rangle\right] \\ &
\leq -\beta V(x)\sum_{\varepsilon\in\{-1,1\}^d}\sum_{j=1}^r \left(\frac{\underline{\lambda}\|v_j\|^2/2}{(1-\langle \varepsilon
    v_j,x\rangle)^2}-\sum_{i=1}^d \frac{\varepsilon_i v_{ji} b_i(x)}{1-\langle \varepsilon v_j,x\rangle}\right)\\
&\leq -\frac{\beta\underline{\lambda}}{4} V(x) \sum_{\varepsilon\in\{-1,1\}^d}\sum_{j=1}^r \frac{\|v_j\|^2}{(1-\langle \varepsilon
  v_j,x\rangle)^2}+B \\ & \leq -\frac{\beta\underline{\lambda}}{4 r2^d} V(x)\|f(x)\|^2+B
\end{align*}
for some positive constant $B>0$. Moreover, choosing $\delta=\frac{\beta}{\alpha-\beta}$, one has
$\frac{V(x)^{1+\delta}}{\varphi(x)^\delta}=1$ and hence
\begin{align*}
LV(x)+\frac{V(x)^{1+\delta}}{\varphi(x)^\delta}\leq  -\frac{\beta\underline{\lambda}}{4 r2^d} V(x)\|f(x)\|^2+B+1,
\end{align*}
where $-\frac{\beta\underline{\lambda}}{4 r2^d} V(x)\|f(x)\|^2+B+1$ is non-positive in a vicinity of the boundary $\d E'$ since
$V(x)\geq d(x,\d E')$ as was proved above. Since $\varphi$ is uniformly bounded from below by a positive constant on any compact
subset of $E'$ and is positive on $E'$, we deduce that $LV(x)+\frac{V(x)^{1+\varepsilon}}{\varphi(x)^\varepsilon}$ is smaller that
$c''\varphi(x)$ for some constant $c''>0$. As a consequence, the right hand side of~\eqref{eq:proof-conclusion} holds true, which
concludes the proof when $b_i(x)=0$ whenever $x_i=0$.

It only remains to extend the last result to cases where $b_i(x)$ does not vanish when $x_i=0$. Since by symmetry the functions $V$
and $\varphi$ satisfy Neumann's boundary condition on the reflecting boundary of $E$, they both belong to the domain of the generator
of the process $X$, hence we can reproduce the computations above which are actually valid for any bounded measurable $b$ and
conclude using the same criterion. This concludes the proof of Theorem~\ref{thm:QSD-diff}.
\end{proof}

We can deduce from~\eqref{eq:expo-conv-diff} similar estimates as in Proposition~\ref{prop:2}, with $\gamma$ playing the role of
$\lambda_1-\lambda_0$ and with the additional difficulty that we cannot obtain pointwise estimates on $\eta$ and $\alpha$, but only
estimates on their mean values on small balls. Again, since $\eta$ is bounded, these estimates are good when $\lambda_0\ll\gamma$.

\begin{prop}
  \label{prop:diffusion}
   For all $t\geq 0$ and $x\in E$,
   \begin{equation}
     \label{eq:estimate-eta-2-bis}
     \eta(x)\leq e^{\lambda_0 t}+Ce^{-\gamma t}
   \end{equation}
   and, for all measurable $A\subset E$ such that $\alpha(A)>0$,
  \begin{equation}
    \label{eq:estimate-eta-2}
    \frac{1}{\alpha(A)}\int_A \eta(x) d\alpha(x)\geq e^{\lambda_0 t}+Ce^{-\gamma t}-\frac{e^{\lambda_0 t}-1+Ce^{-\gamma t}}{\alpha(A)}.
  \end{equation}
  In addition, for all $B\subset E$ measurable such that $\alpha(B)>0$, we define the probability measure $\alpha_{| B}$ as
  $\frac{\mathbbm{1}_{x\in B}d\alpha (x)}{\alpha(B)}$. Then, for all such $B$, all $t\geq 0$ and all $A\subset
  E$ measurable, we have
  \begin{multline}
    \left|\mathbb{P}_{\alpha_{|B}}(X_t\in A)-\alpha(A)\right|\leq \frac{\alpha(A)}{\alpha(B)}\left(1-e^{-\lambda_0
      t}+Ce^{-(\lambda_0+\gamma) t}\right) \\ +Ce^{-(\lambda_0+\gamma) t}\left(\|\eta\|_\infty+2\alpha(A)+ Ce^{-\gamma t}\right).
    \label{eq:estimate-QSD-2}
  \end{multline}
\end{prop}

\begin{proof}
  As explained above, the property~\eqref{eq:expo-conv-diff} implies (see~\cite[Proposition~2.3]{CV}) that
  \begin{equation}
    \label{eq:eta-diff}
    |e^{\lambda_0 t}\mathbb{P}_x(t<\tau_\d)-\eta(x)|\leq C' e^{-\gamma' t},\quad \forall x\in E,\ \forall t\geq 0,
  \end{equation}
  for constants $C'$ and $\gamma'$ that can be assumed without loss of generality equal to th constants $C$ and $\gamma$
  of~\eqref{eq:expo-conv-diff}. In addition, $\int_E \eta(x)\alpha(dx)=1$. 

  We first deduce from~\eqref{eq:eta-diff} the inequality~\eqref{eq:estimate-eta-2-bis}. Combining this with $\alpha(E)=\int_E \eta
  d\alpha=1$, we deduce that
  \begin{align*}
    \int_A \eta(x) d\alpha(x) & =1-\int_{E\setminus A} \eta(x) d\alpha(x) \\
    & \geq 1-(e^{\lambda_0 t}+Ce^{-\gamma t})\alpha(E\setminus A) \\
    & \geq 1-e^{\lambda_0t}-Ce^{-\gamma t}+\alpha(A) (e^{\lambda_0 t}+Ce^{-\gamma t}).
  \end{align*}
  This entails~\eqref{eq:estimate-eta-2}.

  To obtain~\eqref{eq:estimate-QSD-2}, we combine~\eqref{eq:expo-conv-diff},~\eqref{eq:eta-diff} and~\eqref{eq:estimate-eta-2} in
  \begin{align*}
     \left|\mathbb{P}_{\alpha_{|B}}(X_t\in A)-\alpha(A)\right| 
    & \leq \left|\mathbb{P}_{\alpha_{|B}}(X_t\in A\mid t<\tau_\partial)-\alpha(A)\right|\mathbb{P}_{\alpha_{|B}}(t<\tau_\partial) \\
    & +\alpha(A)\left|\mathbb{P}_{\alpha_{|B}}(t<\tau_\partial)-\frac{1}{\alpha(B)}\int_B \eta(x)d\alpha(x)
      e^{-\lambda_0 t}\right| \\ & +\alpha(A)\left|\frac{1}{\alpha(B)}\int_B
      \eta(x)da(x) e^{-\lambda_0 t}-1\right| \\
    & \leq C e^{-\gamma t}\left(\frac{1}{\alpha(B)}\int_B\eta(x)d\alpha(x) e^{-\lambda_0 t}+C
      e^{-(\lambda_0+\gamma)t}\right) \\ & +2C\alpha(A) e^{-(\lambda_0+\gamma)t}+\frac{\alpha(A)}{\alpha(B)}\left(1-e^{-\lambda_0
        t}+Ce^{-(\lambda_0+\gamma)t}\right). \qedhere
  \end{align*}
\end{proof}


\section{Simulations}
\label{sec:simulations}
In this section, we focus on the examples presented in Sections~\ref{sec:colliding-stacks} and~\ref{banker}.\\

As illustration of the different models studied above, we present simulation results of the density of the QSD as well as the value of
$\lambda_{0}$ for a diffusion model of two colliding stacks and a discrete state-space model of a 2-dimensional banker algorithm. The
numerical method applies to general models. We restrict to the case of two stacks and two consumers for pedagogical reasons.

\label{sec:simul}


We use a specific class of particle systems with singular interaction, called Fleming-Viot particle systems, which arises in the
study of distributions of absorbed Markov processes conditioned to non-absorption~\cite{Burdzy,DelMoral}. In these systems,
particles move independently following the Markovian dynamics of the underlying process (here, colliding stacks or banker
algorithms), until one gets absorbed, in which case it is immediately sent to the position of another particle, chosen uniformly at
random. This method overcomes the problematic and necessary deterioration of classical Monte-Carlo techniques in the setting of
absorbed processes, by maintaining a constant sample size of significant particles. 

It is known in general that this method allows to approximate conditional distributions of the underlying Markov process in the limit of
infinitely many particles~\cite{PhD-DV,OV}. In our case, since we want to approximate the quasi-stationary distribution, we need to
simulate the Fleming-Viot process for sufficiently long time. In practice, we compute the ergodic mean of the simulated system and
stop the simulation when its variation goes below a threshold.

In this setting, the eigenvalue $\lambda_0$ is obtained as the average rate of absorption of particles. This approximation relies on
the unbiased estimator introduced in the proof of Theorem~2.1 of~\cite{Villemonais2014}.

\subsection{Simulations for two colliding stacks}

We simulate a diffusive model of two colliding stacks, given by the solution $X_t=(X^{(1)}_t,X^{(2)}_t)$ of the stochastic
differential equation
\begin{align*}
  dX^{(1)}_t & = dB^{(1)}_t+\sigma d B^{(2)}_t \\  
  dX^{(2)}_t & =\sigma dB^{(1)}_t+d B^{(2)}_t,
\end{align*}
with a parameter $\sigma\geq 0$ and two independent Brownian motions $B^{(1)}$ and $B^{(2)}$. The process is assumed to be killed
when $X^{(1)}_t+X^{(2)}_t=1$ and reflected when $X^{(1)}_t=0$ or $X^{(2)}_t=0$. The case $\sigma=0$ corresponds to two independent
Brownian motions for each coordinate. The parameter $\sigma$ governs the correlation between the two coordinates.

\begin{figure}[h]
  \centering
  \mbox{\subfigure[\small Simulation with $\sigma=0$. We obtain $\lambda_0 \approx 4.65$.]%
{\includegraphics[width=0.45\textwidth]{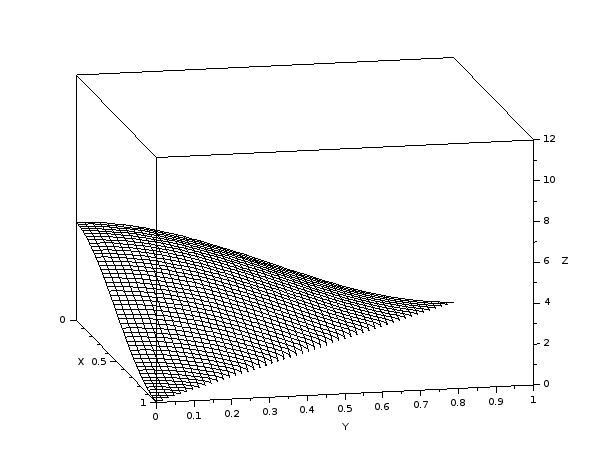}}\quad
    \subfigure[\small Simulation with $\sigma=0.5$. We obtain $\lambda_0 \approx 5.69$.]%
{\includegraphics[width=0.45\textwidth]{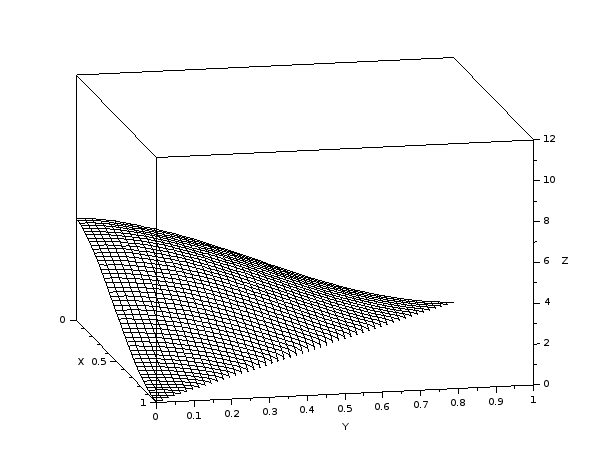}}} \\
  \mbox{\subfigure[\small Simulation with $\sigma=0.7$. We obtain $\lambda_0 \approx 6.73$.]%
{\includegraphics[width=0.45\textwidth]{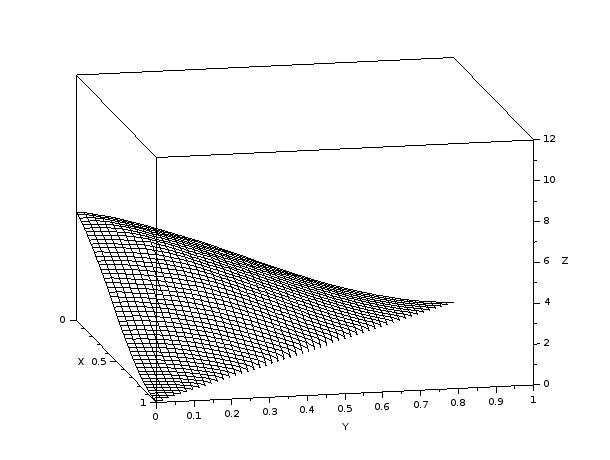}}\quad
    \subfigure[\small Simulation with $\sigma=0.9$. We obtain $\lambda_0 \approx 8.06$.]%
{\includegraphics[width=0.45\textwidth]{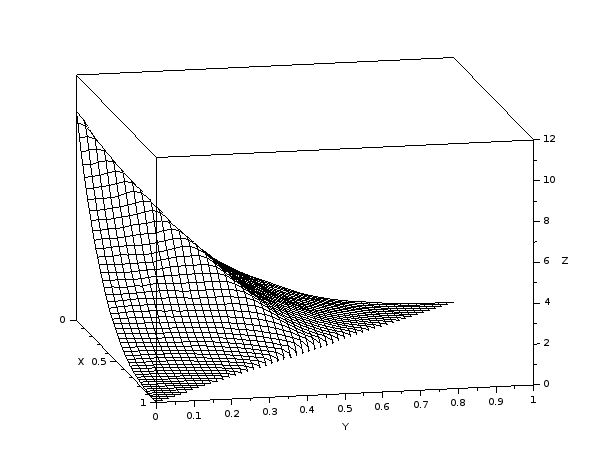}}}
\caption{{\small Density of the QSD for the SDE $dX_t=S dB_t$ in $\mathbb{R}^2$, with
      $S=(\sigma,1;1,\sigma)$, killed when $X^{(1)}_t+X^{(2)}_t=1$ and reflected
      when $X^{(1)}_t=0$ or $X^{(2)}_t=0$.}}
  \label{fig:colliding-stacks}
\end{figure}

The numerical results are presented in Fig.~\ref{fig:colliding-stacks}, where the density of the QSD is plotted for different values
of $\sigma$. We observe that small values of $\sigma$ have little influence on the density of the QSD
(Fig.~\ref{fig:colliding-stacks}(a) and~(b)). Larger values of $\sigma$ have a tendency to concentrate the density close to the line
$x=y$. This is due to the stronger correlation between the two coordinates (with the limiting case $\sigma=1$ where
$X^{(1)}=X^{(2)}$). In all simulations, the density of the QSD vanishes at the absorbing boundaries of the domain, as expected, and
decreases with respect to $X^{(1)}$ and $X^{(2)}$. Hence the larger density is obtained at the point $(0,0)$, farthest from the
absorbing boundary.

The computed eigenvalues $\lambda_0$ increase with $\sigma$. This means that the deadlock of the system is faster for larger values
of $\sigma$. This is due to the stronger correlation between the two coordinates, which makes the process move preferentially and
faster in the direction orthogonal to the absorbing boundary. In the extreme case $\sigma=1$, the process only moves in this
direction and does not explore the major part of the domain. Although the densities of the QSD for $\sigma=0$ and $\sigma=0.5$ are
quite similar, the values of $\lambda_0$ are significantly different.

\subsection{Simulations for the banker algorithm in the case of two consumers}

We consider here a Markov chain on $\mathbb{N}^2$, which is a random walk $X_n=(X^{(1)}_n,X^{(2)}_n)$ in discrete time, with the
following probabilities of transitions. The process jumps from $(i,j)$ to
\begin{align*}
  (i,j) & \text{ with probability } \frac{1}{3}\\
  (i+1,j+1) & \text{ with probability } \frac{1}{6}\\
  (i+1,j-1) & \text{ with probability } \frac{1}{6}\\
  (i-1,j+1) & \text{ with probability } \frac{1}{6}\\
  (i-1,j-1) & \text{ with probability } \frac{1}{6},
\end{align*}
with additional killing of the process when it reaches $X^{(1)}_n+X^{(2)}_n=100$ and with various schemes of reflection. In all
cases, the first (resp. second) coordinate of the process is reflected when $X^{(1)}_n=0$ (resp. $X^{(2)}_n=0$). In addition to this,
reflection occurs when $X^{(1)}_n$ or $X^{(2)}_n$ reach positive thresholds $m_1$ and $m_2$ which vary in the following
simulations. Note that if $m_1\geq 100$ and $m_2\geq 100$, we are back to a model of two colliding stacks. We added a possibility for
the process to remain at the same position to avoid periodicity problems.

The numerical results are presented in Fig.~\ref{fig:banker-algo}, where the QSD is plotted for different values of $m_1$ and $m_2$.
This corresponds to various shapes of the domain. We observe again that the QSD vanishes at the absorbing boundary and increases with the
distance to this boundary. In the three simulations, we observe positive values of the QSD at reflection boundaries.

The computed eigenvalues $\lambda_0$ depend both on the shape of the domain of the process (the larger the domain is, the smaller
$\lambda_0$ should be) and the size of the aborbing part of the boundary (the larger it is, the larger $\lambda_0$ should be). Hence,
the dependence of $\lambda_0$ with respect to $m_1$ and $m_2$ is non-trivial. For example, when comparing
Fig.~\ref{fig:banker-algo}~(c) to~(a), we see that the domain is larger but also the absorbing boundary. In this case, this produces
a larger value of $\lambda_0$, hence a higher speed of deadlock, although the amount of available resources is larger.

\pagebreak

\begin{figure}[h]
  \centering
  \mbox{\subfigure[\small Reflection when
      $X_1\in\{0,80\}$ or $X_2\in\{0,80\}$. We obtain
    $\lambda_0 \approx 2.88\times 10^{-4}$.]%
{\includegraphics[width=0.45\textwidth]{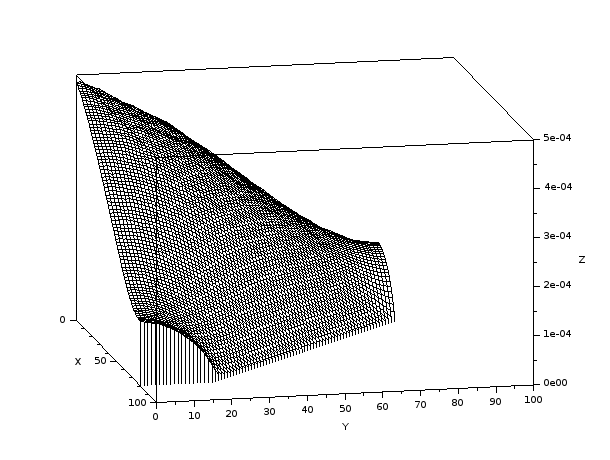}}\quad
    \subfigure[\small Reflection when
      $X_1\in\{0,70\}$ or $X_2\in\{0,90\}$. We obtain
    $\lambda_0 \approx 2.67\times 10^{-4}$.]%
{\includegraphics[width=0.45\textwidth]{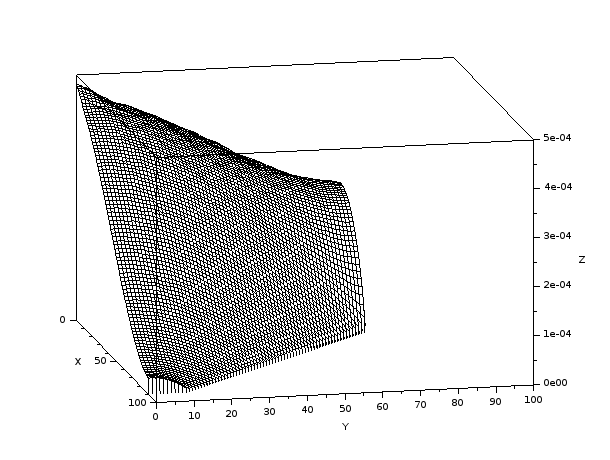}}} \\
  \mbox{\subfigure[\small Reflection when
      $X_1\in\{0,80\}$ or $X_2\in\{0,100\}$. We obtain
    $\lambda_0 \approx 2.99\times 10^{-4}$.]%
{\includegraphics[width=0.45\textwidth]{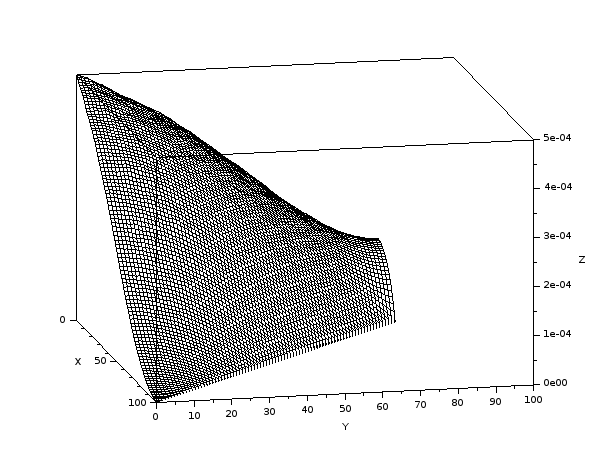}}
}
\caption{{\small Density of the QSD for a symmetric random walk on
      $\mathbb{N}^2$, killed when $X^{(1)}_n+X^{(2)}_n=100$ with different schemes of reflection.}}
  \label{fig:banker-algo}
\end{figure}

\section{Concluding remarks and further research aspects}
\label{sec:ccl}
As far as we know, this paper is the first attempt to use the theory of quasi-invariant distributions for analysing
distributed algorithms which involve large numbers of processors and resources (i.e. types of memories). Our results are 
non-asymptotic in time and in the number of
resources and processors: they are true for any finite values
of the involved parameters and reflect the real behavior of these algorithms.\\
Analysing other types of dynamic algorithms with similar QSD-tools will be the object of further research.

\end{document}